\documentclass[abstracton,a4paper]{scrartcl}

\usepackage[headsepline]{scrpage2}
\usepackage[latin1]{inputenc}
\usepackage{amsfonts}
\usepackage{amsmath}
\usepackage{amssymb}
\usepackage{amsthm}

\usepackage{url}

\usepackage{graphicx}
\usepackage[pdftex,colorlinks,breaklinks,linkcolor=black,citecolor=black,filecolor=black,menucolor=black,urlcolor=black,pdfauthor={Lizhen Ji and Andreas Weber},pdftitle={}, plainpages=false, pdfpagelabels, bookmarksnumbered=true]{hyperref}

\ihead{L.Ji, A. Weber: Dynamics of the heat semigroup}
\newtheorem{theorem}{Theorem}[section]

\newtheorem{proposition}[theorem]{Proposition}
\newtheorem{corollary}[theorem]{Corollary}

\theoremstyle{definition}
\newtheorem{definition}[theorem]{Definition}

\theoremstyle{remark}
\newtheorem{remark}[theorem]{\bf Remark}

\newcommand{\N}{\mathbb{N}}
\newcommand{\Q}{\mathbb{Q}}
\newcommand{\R}{\mathbb{R}}

\newcommand{\C}{\mathbb{C}}

\newcommand*\e{\mathrm{e}}

\newcommand*\re{\mathrm{Re}}
\newcommand*\im{\mathrm{Im}}

\newcommand*\grad{\mathrm{grad}}
\newcommand*\dive{\mathrm{div}}

\newcommand*\isom{\mathrm{Isom}}

\newcommand*\dom{\mathrm{dom}}
\newcommand*\tr{\mathrm{tr}}
\newcommand*\spa{\mathrm{span}}

\newcommand\ad{\mathrm{ad}}

\newcommand*\DMp{\Delta_{M,p}}                 %%%%Laplace-Operatoren

\newcommand*\DM{\Delta_M}

\newcommand*\DXp{\Delta_{X,p}}

\newcommand*\DXps{\Delta^{\#}_{X,p}}

\newcommand*\eDXps{\e^{-t(\Delta^{\#}_{X,p}-c)}} %%%%%% Semigroups
\newcommand*\LpsX{L^p_{\#}(X)}

\title{Dynamics of the heat semigroup on symmetric spaces}
\author{Lizhen Ji\footnote{Email: lji@umich.edu, 
					Address: 1834 East Hall, Ann Arbor, MI 48109-1043, USA.
		     Partially supported by NSF grant DMS 0604878}\\ 
					{\large Department of Mathematics, University of Michigan}
\and Andreas Weber\footnote{    Email: andreas.weber@math.uni-karlsruhe.de,
						   Address:  Englerstr. 2, 76128 Karlsruhe, Germany.}\\
						   {\large  Institut f\"ur Algebra und Geometrie,
						   Universit\"at Karlsruhe (TH)} }

\date{}

\begin{document}

\maketitle 
%%%%%%%%ABSTRACT
\begin{abstract} The aim of this paper is to show that the dynamics of  $L^p$ heat semigroups ($p>2$) on a symmetric space  of non-compact type is very different from the dynamics of the $L^p$ heat semigroups if $1 < p\leq 2$. To see this, it is shown that certain shifts of the $L^p$ heat semigroups have a chaotic behavior if $p>2$, and that such a behavior is not possible in the cases $1 < p\leq 2$.
These results are compared with the corresponding situation for euclidean spaces and symmetric spaces of compact type, where such a behavior is not possible.\\

\noindent{\em Keywords:} Symmetric spaces of non-compact type, $L^p$ heat semigroup,
         spherical functions, Fourier inversion formula, chaotic semigroups.\\[1mm]
       %  {\em MSC 2000:} 
\end{abstract}		
%%%%%%%%%%%%%%

%%%%%%%%%%%%%%%%%%
%% Introduction
%%%%%%%%%%%%%%%%%%
\section{Introduction}

The aim of this paper is to show that the dynamics of  $L^p$ heat semigroups ($p>2$) on 
a symmetric space $X$ of non-compact type is very different from the dynamics of the $L^2$ heat semigroup or, more generally, of the $L^p$ heat semigroups if $1 < p\leq 2$. To see this, we prove that certain shifts of the  $L^p$ heat semigroups ($p>2$) have a chaotic behavior, whereas this is not true for any shift in the  $L^2$ case or, more generally, in the $L^p$ case if $1< p\leq 2$.\\
The different dynamics of the $L^2$ heat semigroup and the $L^p$ heat semigroups is due to
the geometry of the symmetric spaces of non-compact type, and it contrasts the fact that the corresponding
$L^p$ heat semigroups on a euclidean space or on a symmetric space of compact type do not have such a chaotic behavior.\\
A related phenomenon is the $p$-dependence of the $L^p$ spectrum of a symmetric space of non-compact type $X$. While the $L^p$ spectrum of the Laplace operator on euclidean space $\R^n$
equals $[0,\infty)$ for all $p\in [1,\infty)$, the $L^p$ spectrum of the Laplace-Beltrami operator on
$X$ varies with $p$.\\
Actually, this fact together with further information on $L^p$ eigenvalues of the Laplace-Beltrami
operator on $X$ will be used in the proof of the results indicated above. \\
The first systematic study of chaotic semigroup started with the paper \cite{MR1468101}
by Desch, Schappacher, and Webb in 1997 after some papers treating the discrete case 
appeared, cf. for example \cite{MR884467,MR1111569}. 
Further information could be found in  
\cite{MR2150930,MR2128736,MR1855839,MR2020387,MR1468101,MR1685272,MR1484673}.
From these papers we want to mention in particular the paper by Herzog \cite{MR1484673}
and the one by deLaubenfels, Emamirad, and Grosse-Erdmann \cite{MR2020387}
where the heat semigroup is studied in a different context.\\
However, it seems that the theory of chaotic dynamical systems in the sense of 
Devaney as it is used in this paper has not been applied in  'geometric'  settings up to now. \\
Our main input from the papers cited above will be two abstract results contained in  
\cite{MR1468101} and  \cite{MR2128736} which
yield a sufficient condition for chaotic behavior. Further important tools from the realm of symmetric spaces that will be needed are the spherical Fourier transform and the corresponding $L^p$ inversion formula. \\   
Similar results as in the case of symmetric spaces of non-compact type are true for non-compact locally symmetric spaces with finite volume, cf. \cite{JiWeb:nr}.
%%%%%%%%%%%%%%%%%
%% Preliminaries
%%%%%%%%%%%%%%%%%

\section{Preliminaries}

%%%%%%%%%%%%%%%
\subsection{Heat semigroup on $\boldsymbol{L^p}$ spaces}\label{heat semigroup}

In this section $M$ denotes an arbitrary complete Riemannian manifold. The Laplace-Beltrami
operator $\DM := -\dive(\grad)$ with domain $C_c^{\infty}(M)$ (the set of infinitely differentiable functions 
with compact support) is essentially self-adjoint and hence, its closure $\Delta_{M,2}$ is a self-adjoint operator on the Hilbert space $L^2(M)$. Since $\Delta_{M,2}$ is positive,
$-\Delta_{M,2}$ generates a bounded analytic semigroup $\e^{-t\Delta_{M,2}}$ on $L^2(M)$ which can be defined
by the spectral theorem for unbounded self-adjoint operators. The semigroup $\e^{-t\Delta_{M,2}}$
is a {\em submarkovian semigroup} (i.e., $\e^{-t\Delta_{M,2}}$ is positive and a contraction 
on $L^{\infty}(M)$ for any $t\geq 0$) and we therefore have the following:
\begin{itemize}
\item[(1)] The semigroup $\e^{-t\Delta_{M,2}}$ leaves the set $L^1(M)\cap L^{\infty}(M)\subset L^2(M)$ 
		invariant and hence,
		$\e^{-t\Delta_{M,2}}|_{L^1\cap L^{\infty}}$ may be extended to a positive contraction semigroup
		$T_p(t)$ on $L^p(M)$ for any $p\in [1,\infty]$.
		 These semigroups are strongly continuous if $p\in [1,\infty)$ and {\em consistent}
		 in the sense that $T_p(t)|_{L^p\cap L^q} = T_q(t)|_{L^p\cap L^q}$. 
\item[(2)] Furthermore, if $p\in (1,\infty)$, the semigroup $T_p(t)$ is a bounded analytic semigroup
		with angle of analyticity $\theta_p \geq \frac{\pi}{2} - \arctan\frac{|p-2|}{2\sqrt{p-1}}$.	
\end{itemize} 
For a proof of (1) we refer to \cite[Theorem 1.4.1]{MR1103113}. For (2) see \cite{MR1224619}.
In general, the semigroup $T_1(t)$ needs not be analytic. However, if $M$ has bounded geometry
$T_1(t)$ is analytic in {\em some} sector (cf. \cite{MR924464,MR1023321}).

In the following, we denote by $-\DMp$ the generator of $T_p(t)$ and by $\sigma(\DMp)$ the spectrum of $\DMp$. Furthermore, we will write
$\e^{-t\DMp}$ for the semigroup $T_p(t)$.
Because of (2) from above,  the $L^p$ spectrum $\sigma(\DMp)$ has to be contained in the sector
\begin{multline*}
\left\{ z\in \C\setminus\{0\} : |\arg(z)| \leq \frac{\pi}{2}-\theta_p\right\}\cup\{0\} \subset\\
     \left\{ z\in \C\setminus\{0\} : |\arg(z)| \leq \arctan\frac{|p-2|}{2\sqrt{p-1}} \right\}\cup\{0\}.
\end{multline*}     

If we identify as usual the dual space of $L^p(M), 1\leq p<\infty$, with 
$L^{p'}(M), \frac{1}{p}+\frac{1}{p'}=1$, the dual operator of $\DMp$ equals $\Delta_{M,p'}$
and therefore we always have $\sigma(\DMp) = \sigma(\Delta_{M,p'})$.

%%%%%%%%
\subsection{Chaotic semigroups}

There are many different definitions of chaos. We will use the following one which is
basically an adaption of Devaney's definition \cite{MR1046376} to the setting of strongly
continuous semigroups, cf. \cite{MR1468101}.

\begin{definition}
 A strongly continuous semigroup $\{T(t) : t\geq 0\}$ on a Banach space ${\cal B}$ is called {\em chaotic}
 if the following two conditions hold:
   \begin{itemize}
    \item[\textup{(i)}] There exists an $f\in {\cal B}$ such that its orbit 
    			      $\{T(t)f : t\geq 0 \}$ is dense in ${\cal B}$.	 (Then the semigroup is called {\em hypercyclic}.)
    \item[\textup{(ii)}] The set of periodic points
    			        $\{ f\in {\cal B} : \exists t>0 \mbox{~such that~} T(t)f=f \}$	 is dense in ${\cal B}$.		   
   \end{itemize}
\end{definition}

\begin{remark}\label{remark1}
 \begin{itemize}
  \item[\textup{(1)}] As with $\{T(t)f : t\geq 0 \}$  also the set $\{T(q)f : q\in\Q_{\geq 0} \}$
  				is dense, ${\cal B}$ is necessarily separable.
  \item[\textup{(2)}] The orbit of any point $T(t_0)f$ in a dense orbit $\{T(t)f : t\geq 0 \}$ is 
  				again dense in ${\cal B}$. Hence, the set of 
				points with a dense orbit is a dense subset of ${\cal B}$.
  \item[\textup{(3)}] For a separable Banach space ${\cal B}$ condition (i) in the definition 
  		above is equivalent to {\em topological transitivity}
  				of the semigroup $T(t)$, which means that for any pair of non-empty open subsets
				${\cal U,V} \subset {\cal B}$ there is a $t>0$ with
				$T(t){\cal U}\cap {\cal V}\neq \emptyset$, cf. \cite{MR1468101}.
   \item[\textup{(4)}] If both subsets 
   	$${\cal B}_0 = \left\{f\in {\cal B} : T(t)f\to 0\,\,(t\to \infty) \right\}$$ 
	and
	$${\cal B}_{\infty} = 
	\left\{f\in {\cal B} :  \forall \varepsilon >0\, \exists g\in{\cal B} \mbox{~such that~} 
				||g|| <\varepsilon, ||T(t)g -f||<\varepsilon \right\}$$
			 are dense in ${\cal B}$, the semigroup $T(t)$ has dense orbits.
			 However, this condition is not necessary, cf. \cite{MR1468101}.
	For a more detailed study of sufficient conditions for a strongly 
	continuous semigroups to be hypercyclic we refer to	
	\cite{MR2142183,MR2203889,MR2152389,MR2280828,MR2231625,MR2358980}.
   \item[\textup{(5)}] Hypercyclic semigroups exist only on infinite dimensional Banach spaces. 
   			       When looking at the Jordan canonical form of a bounded operator on
			       a (real or complex) finite dimensional Banach space, a proof of this
			       is straightforward, see e.g. \cite[Proposition 11]{MR1685272}.
		  On the other hand, on every infinite dimensional separable Banach space
		  exists a hypercyclic semigroup, cf.  \cite{MR1974641}.
   \item[\textup{(6)}] In Devaney's definition of chaotic dynamical systems it was also required that 
   	orbits depend sensitively on initial conditions. This property is redundant in this setting, see e.g.
   	       \cite{MR1157223,MR1111569}.
   \item[\textup{(7)}] For results concerning the discrete case we refer to 
   		\cite{MR2079359,MR1710637,MR1111569,MR1685272,MR2041336}.
  \end{itemize}
\end{remark}

A sufficient condition for a strongly continuous semigroup to be chaotic in terms of spectral properties
of its generator was given by Desch, Schappacher, and Webb: 
\begin{theorem}[\cite{MR1468101}] \label{thm dsw}
Let $T(t)$ denote a strongly continuous semigroup on a separable 
 Banach space ${\cal B}$ with generator $A$. Assume there is an open, connected subset
 $\Omega\subset \sigma_{pt}(A)$ of the point spectrum of $A$ and a function $F: \Omega\to {\cal B}$ such that
 \begin{itemize}
  \item[\textup{(i)}] $\Omega \cap i\R \neq \emptyset$.
  \item[\textup{(ii)}] $F(\lambda) \in \ker(A-\lambda)$ for all $\lambda \in \Omega$.
  \item[\textup{(iii)}] For all $\phi \in {\cal B'}$ in the dual space of ${\cal B}$, the mapping
  				$F_{\phi}:\Omega\to \C,\, \lambda\mapsto \phi\circ F $
				is analytic. 
				Furthermore, if for some $\phi \in {\cal B'}$ we have $F_{\phi}=0$
				then already $\phi = 0$ holds.
 \end{itemize}
 Then the semigroup $T(t)$ is chaotic.
\end{theorem}
In  \cite{MR1468101} it was also required that the elements $F(\lambda)$, $\lambda \in \Omega$, are non-zero, but as remarked in \cite{MR2128736} this assumption is redundant. \\
In order to make this paper more comprehensive, we include the idea of the proof.
\begin{proof}
 A major role in the proof is played by the following observation: let $U\subset \Omega$
 be any subset that contains an accumulation point. Then it follows  that the
 subset ${\cal B}_{U} = \spa\{ F(\lambda) : \lambda\in U\}$   is dense in ${\cal B}$.
 Indeed, if we suppose the contrary, by the Hahn-Banach Theorem 
 there exists some $\phi\in {\cal B'}, \phi\neq 0,$ such that $\phi\circ F(\lambda)=0$ for all
 $\lambda\in U$. As $U$ contains an accumulation point, it follows from the identity
 theorem for complex analytic functions that $F_{\phi}=0$. But this is a contradiction.\\
 For the subsets $U_0 = \{\lambda\in\Omega : \re(\lambda) < 0\}$,
 $U_{\infty} = \{ \lambda\in\Omega : \re(\lambda) > 0\}$, and
 $U_{per} = \Omega \cap i\Q$ it follows now
 ${\cal B}_{U_0} \subset {\cal B}_0$, ${\cal B}_{U_{\infty}} \subset {\cal B}_{\infty}$, and
 ${\cal B}_{U_{per}} \subset \{ f\in {\cal B} : \exists t>0 \mbox{~such that~} T(t)f=f \}$.
 As all these sets are dense in ${\cal B}$, the proof is complete.
\end{proof}

In the theory of dynamical systems chaotic semigroups are highly unwanted because of their
difficult dynamics. Not much more appreciated are so called subspace chaotic semigroups.
This notion was introduced in \cite{MR2128736}:
\begin{definition} 
  A strongly continuous semigroup $T(t)$ on a Banach space ${\cal B}$ is called {\em subspace 
  chaotic} if there is a closed, $T(t)$ invariant subspace ${\cal V}\neq \{0\}$ of  ${\cal B}$ such that
  the restriction $T(t)|_{\cal V}$ is a chaotic semigroup on ${\cal V}$.
\end{definition}
Because of Remark \ref{remark1} such a subspace is always infinite dimensional.
Banasiak and Moszy\'nski showed that a subset of the conditions in Theorem \ref{thm dsw} yield a sufficient  condition
for subspace chaos:
\begin{theorem}\textup{(\cite[Criterion 3.3]{MR2128736}).}\label{thm ban}
Let $T(t)$ denote a strongly continuous semigroup on a separable 
 Banach space ${\cal B}$ with generator $A$. Assume there is an open, connected subset
 $\Omega\subset \C$ and a function $F: \Omega\to {\cal B}, F\neq 0,$ such that
 \begin{itemize}
  \item[\textup{(i)}] $\Omega \cap i\R \neq \emptyset$.
  \item[\textup{(ii)}] $F(\lambda) \in \ker(A-\lambda)$ for all $\lambda \in \Omega$.
  \item[\textup{(iii)}] For all $\phi \in {\cal B'}$ in the dual space of ${\cal B}$, the mapping
  				$F_{\phi}:\Omega\to \C,\, \lambda\mapsto \phi\circ F $
				is analytic. 
 \end{itemize}
 Then the semigroup $T(t)$ is subspace chaotic.\\
 Furthermore, the restriction of $T(t)$ to the $T(t)$ invariant subspace 
 ${\cal V} = \overline{\spa F(\Omega)}$ is chaotic.	  
\end{theorem}
Note that it is not required $\Omega\subset \sigma_{pt}(A)$ here, i.e. either $F(\lambda)$ is an eigenvector  or $F(\lambda)=0$ .  But, as explained in \cite{MR2128736}, the assumption $\Omega\subset \C$ is not really weaker.\\
Theorem \ref{thm dsw} and Theorem \ref{thm ban} will be important ingredients in the proof of our
main results in Section \ref{section dynamics}.

%%%%%%%%%%%%%%%%%%%
\subsection{Symmetric spaces}\label{symmetric spaces}

Let $X$ denote a symmetric space of non-compact type. Then
$G:= \isom^0(X)$ is a non-compact, semi-simple Lie group with trivial center 
that acts transitively on $X$ and $X=G/K$, where $K\subset G$ is a maximal 
compact subgroup of $G$. We denote
their respective Lie algebras by $\mathfrak{g}$ and $\mathfrak{k}$. Given a corresponding Cartan
involution $\theta: \mathfrak{g}\to\mathfrak{g}$, we obtain the Cartan decomposition
$\mathfrak{g}=\mathfrak{k}\oplus\mathfrak{p}$ of $\mathfrak{g}$ into the eigenspaces of $\theta$. The subspace
$\mathfrak{p}$ of $\mathfrak{g}$ can be identified with the tangent space $T_{eK}X$. We assume
that the Riemannian metric $\langle\cdot,\cdot\rangle$ of $X$ in $\mathfrak{p}\cong T_{eK}X$ 
coincides with the restriction of the Killing form 
$B(Y,Z) := \tr(\ad Y\circ \ad Z ), Y, Z\in \mathfrak{g},$ to $\mathfrak{p}$. \\
For any maximal abelian subspace $\mathfrak{a}\subset \mathfrak{p}$ we refer to 
$\Sigma=\Sigma(\mathfrak{g},\mathfrak{a})$ as the set of restricted roots for the pair $(\mathfrak{g},\mathfrak{a})$,
i.e. $\Sigma$ contains all $\alpha\in \mathfrak{a}^*\setminus\{0\}$ such that
$$ \mathfrak{g}_{\alpha} := \{ Y\in \mathfrak{g} : \ad(H)(Y) = \alpha(H)Y \mbox{~for all~} H\in\mathfrak{a} \}\neq \{0\}.$$
These subspaces $ \mathfrak{g}_{\alpha}\neq \{0\}$ are called root spaces.\\
Once a positive Weyl chamber $\mathfrak{a}^+$ in $\mathfrak{a}$ is chosen, we denote by
$\Sigma^+$ the  subset of positive roots and by 
$\rho:= \frac{1}{2}\sum_{\alpha\in\Sigma^+} (\dim \mathfrak{g}_{\alpha})\alpha$ 
half the sum of the positive roots (counted according to their multiplicity).\\

A function $f: X\to \C$ is called $K$-{\em invariant} if $f(k\cdot x)= f(x)$ for all $k\in K, x\in X$
holds. As $K$ acts by isometries on $X$, the subspace $\LpsX$ of $K$-invariant functions
in $L^p(X), p\in [1,\infty),$ is closed.
Furthermore, the closed subspace $\LpsX$ is invariant under the action of the semigroup
$\e^{-t\DXp}$ (this follows from the $G$-invariance of the heat kernel).  
As the Laplacian is $G$-invariant, it follows that the restriction $\left.\e^{-t\DXp}\right|_{\LpsX}$ of the $L^p$ heat semigroup to $\LpsX$ is a strongly continuous semigroup with generator 
$-\left.\DXp\right|_{\dom(\DXp)\cap \LpsX} =: -\DXps$ (see \cite[Chapter II.2.3]{MR1721989}). 
In the following we will therefore write $e^{-t\DXps}$ instead of $\left.\e^{-t\DXp}\right|_{\LpsX}$\,. \\

To define a very important class of $K$-invariant functions, we recall the Iwasawa decomposition
$G=NAK$ of the semi-simple Lie group $G$, where $N$ is nilpotent, $A$ an 
abelian subgroup with Lie algebra $\mathfrak{a}$, and $K$ the maximal compact subgroup from above. More precisely, this means that for each $g\in G$
there are unique elements $n\in N,a\in A,$ and $k\in K$ such that $g=nak = n\exp(A(g))k$.
A  {\em spherical function} is a $K$-invariant function of the form
$$
  \varphi_{\lambda}(gK) = \int_K \e^{(i\lambda + \rho)(A(kg))} dk,\qquad \lambda\in\mathfrak{a}^*_{\C}.
$$
Because of the uniqueness of the Iwasawa decomposition it follows that the element 
$A(kg)\in\mathfrak{a}$ is well defined.
\begin{proposition}\label{prop spherical functions}
The spherical functions $\varphi_{\lambda}$ have the following properties:
 \begin{itemize}
  \item[\textup{(a)}] $\varphi_{\lambda}\in C^{\infty}(X).$ 
  \item[\textup{(b)}] $\Delta_X \varphi_{\lambda} = 
  				(\langle\lambda,\lambda\rangle + ||\rho||^2)\varphi_{\lambda}$
  \item[\textup{(c)}] Let $C(\rho)$ denote the convex hull of the points $s\rho\in\mathfrak{a}^*, s\in W$,
  			    where $W$ 
  			    denotes the Weyl group. Then for $p>2$ and any 
			    $\lambda \in \mathfrak{a}^* + i(1-\frac{2}{p})C(\rho)$ the spherical function 
			    $\varphi_{\lambda}$ is contained in $L^p(X)$. 
 \end{itemize}
\end{proposition}
The proof of (a) and (b) can be found in \cite{MR754767}. For a proof of (c) see
\cite[Proposition 2.2]{MR1016445} or \cite{MR954385}. 

Taylor also observed in \cite[Proposition 2.2]{MR1016445} also that from Proposition \ref{prop spherical functions} (b) and (c) we may conclude that each 
$\varphi_{\lambda},\, \lambda\in \mathfrak{a}^* + i(1-\frac{2}{p})C(\rho)$,  is an eigenfunction
of $\DXp$  for $p>2$: From  \cite[Theorem 3.9]{MR705991} it follows 
(put $v(t,x)=\e^{-(\langle\lambda,\lambda\rangle + ||\rho||^2)t}\varphi_{\lambda}$) that
$$
 \e^{-t\DXp}\varphi_{\lambda} = \e^{-(\langle\lambda,\lambda\rangle + ||\rho||^2)t}\varphi_{\lambda},
$$
and hence 
$$
 \DXp \,\varphi_{\lambda} = \left(\langle\lambda,\lambda\rangle + ||\rho||^2\right) \varphi_{\lambda}.
$$ 
From this observation, we obtain that the interior of the parabolic region
\begin{equation}
 P_p = \left\{ ||\rho||^2 + z^2 : z\in\C, |\im z|\leq ||\rho||\cdot |\frac{2}{p}-1|\right\}\subset \C.
\end{equation} 
is contained in the point spectrum $\sigma_{pt}(\DXp)$ of $\DXp$ if $p>2$. \\
Taylor actually proved in \cite{MR1016445} that the spectrum $\sigma(\DXp)$ of $\DXp$ coincides with $P_p$ for all $p\in [1,\infty)$.

Two important tools, that will be used later are the {\em spherical Fourier transform} and its
$L^p$ {\em inversion formula}.

\begin{theorem} Let $p\in [1,2], f\in \LpsX$ and 
$\lambda\in\mathfrak{a}^* + i\left(1-\frac{2}{p}\right)C(\rho)$. Then
$$
 {\cal F}f(\lambda) := \int_X f(x)\varphi_{-\lambda}(x) dx
$$
exists and the map $\lambda\mapsto  {\cal F}f(\lambda)$ is analytic.
\end{theorem}
For the analyticity of  $\lambda\mapsto{\cal F}f(\lambda)$
see \cite[Corollary 8.2]{MR754767} where the $L^1$ case is treated.\\
The $L^p$ inversion formula for the spherical Fourier transform was proved by
Stanton and Tomas (by $c(\lambda)$ we denote Harish-Chandra's $c$ function):

\begin{theorem}\textup{\cite[Theorem 3.3]{MR518528}}\label{thm stantom}
Let $p\in [1,2), f\in \LpsX $ with ${\cal F}f\in L^1(\mathfrak{a}^*_{+},|c(\lambda)|^{-2}d\lambda)$ and
let $x$ be in the Lebesgue set of $f$. Then we have
$$
 f(x) = \int_{\mathfrak{a}^*_{+}} {\cal F}f(\lambda) \varphi_{\lambda}(x) |c(\lambda)|^{-2}d\lambda.
$$
\end{theorem}
Recall, that a point $x\in X$ is said to be contained in the Lebesgue set of $f$ if
$$
  \lim_{r\to 0} |B(x,r)|^{-1}\int_{B(x,r)} |f(y)- f(x)| dy \neq 0,
$$
where $B(x,r)$ denotes the geodesic ball with center $x$ and radius $r>0$.
As remarked in \cite{MR518528}, the Lebesgue set of an $L^p$ function has co-measure zero.

%%%%%%%%%%%%%%%%%
%%%%%%%%%%%%%%%%%
\section{Dynamics of the heat semigroup}\label{section dynamics}

\begin{theorem}\label{thm rank one}
Let $X=G/K$ denote a symmetric space of non-compact type with rank one.
\begin{itemize}
 \item[\textup{(a)}] Let $p>2$. Then there
  is a $c_p>0$ such that for all $c> c_p$ the semigroup
  $$
   \eDXps: \LpsX \to \LpsX
  $$
  is chaotic.
  \item[\textup{(b)}] Let $c\in\R$. Then the semigroup
   $$
    e^{-t(\Delta_{X,2}^{\#} - c)} : L^2_{\#}(X) \to L^2_{\#}(X) 
   $$
   is not chaotic.
  \item[\textup{(c)}] Let $1< p < 2$ and $c\in \R$. Then the semigroup
     $$
       \eDXps: \LpsX \to \LpsX
     $$
     is not hypercyclic (and therefore in particular not chaotic).
\end{itemize}
\end{theorem}

\begin{proof}
 For the proof of part (a) we will check the conditions of Theorem \ref{thm dsw}.\\
 As the rank of $X$ equals one, we have $\mathfrak{a}_{\C} \cong \C$ and this
 identification will be used in what follows. More precisely, any $z\in \C$ will be identified
 with $\lambda\in \mathfrak{a}_{\C}^*$ (and vice versa) where $\lambda(H_0) = z$ for the (unique)
 $H_0$ in $\mathfrak{a}^+$ with $||H_0||=1$.

 Let  $p>2$. As we have seen in Section \ref{symmetric spaces}, 
 the point spectrum of $\DXps$ contains all points in the interior of the parabolic region
 $P_p$ and corresponding eigenfunctions are given by certain spherical functions.
 The apex of the parabolic region $P_p$ is at the point
 $$
  c_p = ||\rho||^2 -  ||\rho||^2\cdot \left|\frac{2}{p}-1\right|^2 = \frac{4||\rho||^2}{p}\left(1-\frac{1}{p}\right)
 $$ 
 and hence, the point spectrum of $(\DXps - c)$ intersects the imaginary axis for all
 $c > c_p$.\\
 Let us assume in the following $c>c_p$ and let $\Omega$ denote the interior of the set
 $$
  \left(P_p -c \right)\setminus \{z\in \R : z\leq ||\rho||^2 - c \}. 
 $$
 Then $\Omega$ is an open, connected subset of the point spectrum of $(\DXps - c)$
 that intersects the imaginary axis and it remains to 
  define an analytic mapping $F: \Omega \to \LpsX$ satisfying (iii) of Theorem \ref{thm dsw}. \\
  For this, we put
  $$
   F(z) = \varphi_{-\sqrt{z + c -||\rho||^2}},\quad (z\in \Omega),
  $$
  where $\varphi_{-\lambda}(gK) = \int_{K}\e^{(-i\lambda + \rho)(A(kg))}dk$ 
  denotes the corresponding spherical function and an analytic branch
  of the square root is chosen (such a choice is possible as by construction
  $(\Omega + c -||\rho||^2) \cap \R_{\leq 0}= \emptyset$). \\
  By the Cartan decomposition $G=K\exp(\overline{\mathfrak{a}^+})K$ of the semi-simple Lie group
  $G$, the closed subspace $\LpsX$ of $K$-invariant functions in $L^p(X)$ can be identified
  with $L^p(\mathfrak{a}^+, d\mu)$, where  $d\mu = \frac{d\lambda}{|c(\lambda)|^2}$ 
  denotes the Plancherel measure. 
  Hence, 
  the dual space of $\LpsX$ can be identified with 
  $L^{p'}(\mathfrak{a}^+, d\mu)\cong L^{p'}_{\#}(X), \frac{1}{p} + \frac{1}{p'}=1$.
  Let $f\in L^{p'}_{\#}(X)$. Then we have
  $$
   F_f: \Omega\to \C,\, z\mapsto \int_X f(x)\varphi_{-\sqrt{z + c -||\rho||^2}}(x) dx
  $$
  and $F_f$ is basically the spherical transform of $f$. More precisely, this means the following.
  The function $z\mapsto \sqrt{z + c -||\rho||^2}$
  maps the set $\Omega$ analytically onto the strip
  $$
  S_p = \left\{z\in \C :  \re(z)>0, |\im z| < ||\rho||\cdot |\frac{2}{p}-1| \right\}.
  $$
  For any $\lambda\in S_p$ the spherical function $\varphi_{-\lambda}$ is contained in $\LpsX$
  (see Section \ref{symmetric spaces}) and
  the $L^{p'}$ spherical transform of a function $f\in L^{p'}_{\#}(X)$ is given by
  $$
   {\cal F} f (\lambda) = \int_X f(x) \varphi_{-\lambda}(x) dx.
  $$
Therefore, we have $F_f(z) = {\cal F}f(\sqrt{z +c -||\rho||^2})$ and as both maps 
$\lambda\mapsto{\cal F}f(\lambda)$ and 
$\sqrt{\hspace{1ex}  }$ are  analytic, also $F_f$ is analytic. \\
Let us now assume $F_f = 0$. Hence, the $L^{p'}$ spherical transform of $f$ satisfies 
${\cal F}f(\lambda) = 0$ for all $\lambda\in S_p$. From the inversion formula  by Stanton and Tomas  
(Theorem \ref{thm stantom}) we may conclude $f=0$ and the claim in (a) follows.\\

The Laplacian $\Delta_{X,2}$ is a selfadjoint operator and
hence,  $\sigma(\Delta_{X,2}^{\#})\subset\R$. But as the intersection of the point
spectrum of a generator of a chaotic semigroup with the imaginary axis   is always
infinite (see \cite{MR1855839} and its erratum) part (b) follows.\\

For the proof of part (c) we proceed as follows. Let $1<p<2$ and
assume that the semigroup $\eDXps$ is hypercyclic. Then, the dual operator
$(\DXps - c)' = \Delta_{X,p'}^{\#} - c$ of its generator would have empty point spectrum
(cf. \cite[Theorem 3.3]{MR1468101}) -- a contradiction. 
\end{proof}

%%%%%%%
\subsection{Higher rank symmetric spaces}

In the higher rank case most parts of the proof of Theorem \ref{thm rank one}, part (a), remain
valid after some adaption.\\
For $p>2$ and any $\lambda \in \mathfrak{a}^* + i(1-\frac{2}{p})C(\rho)$
the spherical function $\varphi_{\lambda}$ is contained in $L^p(X)$, see Section \ref{symmetric spaces}. 

We now replace $\Omega$ by $\Omega'$, defined as the interior of
$$
(P_p-c)\cap \{z\in\C : \im(z)>0 \}.
$$
If we choose the usual analytic branch of the square root, the set $\Omega'$
is mapped by $h(z) = ||\rho||^{-1}\sqrt{z+c-||\rho||^2}$ (analytically) onto the strip
$$
\left\{ z\in\C : \re(z)>0, 0 < \im(z) < (1-\frac{2}{p})\right\}.
$$
Hence, the map 
$$
 F: \Omega' \to \LpsX,\, z\mapsto \varphi_{h(z)\rho}
$$
is analytic. Note, that each spherical function $\varphi_{h(z)\rho}, z\in\Omega'$,
is contained in $L^p(X)$ as both  $0$ and $\rho$ are contained in the convex set
$C(\rho)$, and therefore $h(z)\rho \in \mathfrak{a}^* + i(1-\frac{2}{p})C(\rho)$. \\
 As these spherical functions are eigenfunctions of $\DXps$ all assumptions in 
 Theorem \ref{thm ban} are valid and thus, we have the following:
 \begin{theorem}\label{thm higher rank}
 Let $X$ denote a symmetric space of non-compact type and let $p>2$. Then there is a $c_p > 0$
 such that for any $c > c_p$ the semigroup
 $$
  \e^{-t(\Delta_{X,p}^{\#} - c)}: \LpsX \to \LpsX
 $$
 is subspace chaotic.
 \end{theorem}  
 Since it is impossible in the higher rank case (for dimensional reasons) 
 to define an analytic map from an open subset of 
 $\Omega$ onto a neighborhood of
 $\mathfrak{a}_+^*$ in $\mathfrak{a}^* + i(1-\frac{2}{p})C(\rho)$, we are not able to apply
 the inversion formula  in Theorem \ref{thm stantom}  and therefore this
 method cannot be used to prove that the restriction $\eDXps: \LpsX \to \LpsX$ of a shift of the 
 heat semigroup to the subspace of $K$-invariant $L^p$-functions is chaotic.\\
 From Theorem \ref{thm rank one} and Theorem \ref{thm higher rank} it immediately follows
 \begin{corollary}
  Let $X$ denote a symmetric space of non-compact type and let $p>2$. Then there is a $c_p > 0$
 such that for any $c > c_p$ the semigroup
 $$
  \e^{-t(\Delta_{X,p} - c)}: L^p(X) \to L^p(X)
 $$
 is subspace chaotic. 
 \end{corollary}
 In the rank one case we identified precisely an invariant subspace $\LpsX$ on which 
 $\e^{-t(\Delta_{X,p} - c)}$ is chaotic, in the higher rank case we only know that there is an
 invariant subspace of $\LpsX$ on which the respective semigroup is chaotic.
 However, if $X$ is the product of rank one spaces it can be shown that certain shifts of
 the heat semigroup on $\LpsX, p>2,$ at least have  dense orbits:
 \begin{proposition}\label{products}
 Let $X_1,\ldots, X_k$ denote symmetric spaces of non-compact type with rank one, 
 $X= X_1\times\cdots\times X_k$ their Riemannian product, and let $p>2$. Then there is a 
 $c_p>0$ such that for any $c>c_p$ the semigroup
  $$
  \e^{-t(\Delta_{X,p}^{\#} - c)}: \LpsX \to \LpsX
 $$
 has dense orbits.
 \end{proposition}
\begin{proof}
We restrict ourselves to the case of two factors $X= X_1\times X_2$. The changes below  for the 
case $k>2$ are obvious. In the proof we basically use a result from  \cite{Weber:yq} where also
more general results concerning the dynamics of tensor products of semigroups can be found.\\ 
Note, that if $X_1 = G_1/K_1$ and $X_2=G_2/K_2$ we have 
$X= (G_1\times G_2)/(K_1\times K_2)$.
By $L^p_{\#}(X_1)\otimes L^p_{\#}(X_2)$ we denote the (algebraic) tensor product of 
$L^p_{\#}(X_1)$ and $L^p_{\#}(X_2)$. For the uniform crossnorm $g_p$ on the vector space 
$L^p_{\#}(X_1)\otimes L^p_{\#}(X_2)$ as defined
in \cite[Definition 2.3]{Weber:2008ve} it follows that the completion 
$L^p_{\#}(X_1)\tilde{\otimes}_{g_p}L^p_{\#}(X_2)$ of the normed space
$(L^p_{\#}(X_1)\otimes L^p_{\#}(X_2),g_p)$ coincides with $\LpsX$. 
This follows exactly as in the proof of \cite[Lemma 2.4]{Weber:2008ve} when observing
that the map
$$\begin{array}{rl}
 \iota:  & \left( L^p_{\#}(X_1)\otimes L^p_{\#}(X_2), g_p\right) \to  L^p_{\#}(X_1\times X_2),\vspace{2mm}\\     {}       &  \sum_{j=1}^{n}f_j\otimes g_j\mapsto \sum_{j=1}^{n}f_j g_j
\end{array}$$    
is isometric with dense image. As the subspace $\LpsX$ of $K$-invariant functions is closed in
$L^p(X)$, the claim follows. Furthermore, by the arguments in \cite{Weber:2008ve} 
(cf. also \cite{MR0348538})
we may also conclude that
$$
 \e^{-t\Delta_{X,p}^{\#}} = \e^{-t\Delta_{X_1,p}^{\#}}\otimes \e^{-t\Delta_{X_2,p}^{\#}}.
$$
By Theorem \ref{thm rank one} the semigroups 
$$ 
T_1(t) = \e^{-t(\Delta_{X_1,p}^{\#}-c_1)}, \quad c_1 > c_{p,1},
$$ 
and  
$$
T_2(t) = \e^{-t(\Delta_{X_2,p}^{\#}-c_2)}, \quad c_2 > c_{p,2},
$$ 
are chaotic.
Hence, it follows from  \cite[Corollary 2.2]{Weber:yq} that their tensor product 
$T_1(t)\otimes T_2(t)$ on $ L^p_{\#}(X) = L^p_{\#}(X_1)\tilde{\otimes}_{g_p}L^p_{\#}(X_2)$
is a strongly continuous semigroup that has dense orbits (it is even recurrent hypercyclic).
As
$T_1(t)\otimes T_2(t) = \e^{-t(\Delta_{X,p}^{\#}-c)}, c= c_1 + c_2$, the result follows
if we put $c_p=c_{p,1} + c_{p,2}$.
\end{proof}

\begin{corollary}
 Let $X_1,\ldots, X_k$ denote symmetric spaces of non-compact type with rank one, 
 $X= X_1\times\cdots\times X_k$ their Riemannian product, and let $1<p\leq 2$.
 Then 
$$
 \sigma_{pt}(\Delta_{X,p}^{\#}) =\emptyset.
$$ 
\end{corollary}
This implies in particular that $\DXp$ has no $K$-invariant eigenfunctions in the case 
$1 < p \leq 2$.
\begin{proof}
For $p=2$ this is clearly true as already $\Delta_{X,2}$ has no eigenvalues.
For $1<p<2$ this follows immediately from Proposition \ref{products} and the fact that the dual operator
of the generator of a strongly continuous semigroup with dense orbits has empty point spectrum
(cf. \cite{MR1468101}) when observing that $(\Delta_{X,p}^{\#})' = \Delta_{X,p'}^{\#}.$
\end{proof}

%%%%%%%
\subsection{Comparison with symmetric spaces of euclidean or compact type}

Any globally symmetric space $X$ is the product $X=X_-\times X_0 \times X_+$ of 
a symmetric space of non-compact type $X_-$, a euclidean space $X_0\cong\R^n$, and
a symmetric space of compact type $X_+$ which is always compact (see \cite{MR514561}).
The aim of this section is to show that such a chaotic dynamics, as in the setting of 
symmetric spaces of non-compact type, cannot occur in the two other situations.

\begin{theorem}
 Let $X=X_0\times X_+$ denote a symmetric space whose factors are a symmetric space $X_0$ 
 of euclidean type and a symmetric space $X_+$ of compact type (it is allowed that one of the factors is trivial). Then, for any
 $p\in [1,\infty)$ and any $c\in\R$ the semigroup
 $$
  \e^{-t(\DXp - c)}: L^p(X) \to L^p(X)
 $$
 is not subspace chaotic.
\end{theorem}
\begin{proof}
The intersection of the point spectrum of the generator of a chaotic semigroup with the imaginary axis is always infinite (see \cite{MR1855839} and its erratum).
Hence, the theorem follows if we can prove  
$\sigma(\DXp)\subset \R$. To see this, we note that
$\sigma(\Delta_{X,2}) = [0,\infty)$ if $X$ is a euclidean space and
$\sigma(\Delta_{X,2}) = \{ \lambda_j\in \R : j\in \N\}$ if $X$ is a symmetric space of 
compact type. As the $L^p$ spectrum $\sigma(\DXp)$ coincides in these cases 
with the $L^2$ spectrum (see e.g. \cite{MR1250269}) and as the $L^p$ spectrum 
$\sigma(\Delta_{X_0\times X_+,p})$ of the product $X_0\times X_+$  coincides with the sum 
$\sigma(\Delta_{X_0,p}) + \sigma(\Delta_{X_+,p})$ (see \cite[Theorem 2.1]{Weber:2008ve})
the result follows. 
\end{proof}
%%%%%%%%%%%%%%%%%

\subsubsection*{Acknowledgement}
We want to thank the referee who made many valuable comments that greatly
improved the exposition of this paper.

%%%%%%%%%%%%%%BACKMATTER%%%%%%%%%%%%

\bibliographystyle{amsplain}
%\addcontentsline{toc}{chapter}{Bibliography}
\bibliography{dissertation,hypercyclic,symmetricSpaces}

\providecommand{\bysame}{\leavevmode\hbox to3em{\hrulefill}\thinspace}
\providecommand{\MR}{\relax\ifhmode\unskip\space\fi MR }
\providecommand{\MRhref}[2]{%
  \href{http://www.ams.org/mathscinet-getitem?mr=#1}{#2}
}
\providecommand{\href}[2]{#2}
\begin{thebibliography}{10}

\bibitem{MR2150930}
Richard~M. Aron, Juan~B. Seoane-Sep{\'u}lveda, and Andreas Weber, \emph{Chaos
  on function spaces}, Bull. Austral. Math. Soc. \textbf{71} (2005), no.~3,
  411--415. 

\bibitem{MR2128736}
Jacek Banasiak and Marcin Moszy{\'n}ski, \emph{A generalization of
  {D}esch-{S}chappacher-{W}ebb criteria for chaos}, Discrete Contin. Dyn. Syst.
  \textbf{12} (2005), no.~5, 959--972. 

\bibitem{MR1157223}
J.~Banks, J.~Brooks, G.~Cairns, G.~Davis, and P.~Stacey, \emph{On {D}evaney's
  definition of chaos}, Amer. Math. Monthly \textbf{99} (1992), no.~4,
  332--334. 

\bibitem{MR2142183}
Teresa Berm{\'u}dez, Antonio Bonilla, Jos{\'e}~A. Conejero, and Alfredo Peris,
  \emph{Hypercyclic, topologically mixing and chaotic semigroups on {B}anach
  spaces}, Studia Math. \textbf{170} (2005), no.~1, 57--75. 

\bibitem{MR1974641}
Teresa Berm{\'u}dez, Antonio Bonilla, and Antonio Martin{\'o}n, \emph{On the
  existence of chaotic and hypercyclic semigroups on {B}anach spaces}, Proc.
  Amer. Math. Soc. \textbf{131} (2003), no.~8, 2435--2441 (electronic).
  

\bibitem{MR2079359}
Teresa Berm{\'u}dez, Antonio Bonilla, and Alfredo Peris, \emph{On
  hypercyclicity and supercyclicity criteria}, Bull. Austral. Math. Soc.
  \textbf{70} (2004), no.~1, 45--54. 

\bibitem{MR1710637}
Juan B{\`e}s and Alfredo Peris, \emph{Hereditarily hypercyclic operators}, J.
  Funct. Anal. \textbf{167} (1999), no.~1, 94--112. 

\bibitem{MR2203889}
Jos{\'e}~A. Conejero and Alfredo Peris, \emph{Linear transitivity criteria},
  Topology Appl. \textbf{153} (2005), no.~5-6, 767--773. 

\bibitem{MR1023321}
E.~Brian Davies, \emph{Pointwise bounds on the space and time derivatives of
  heat kernels}, J. Operator Theory \textbf{21} (1989), no.~2, 367--378.
  

\bibitem{MR1103113}
\bysame, \emph{Heat kernels and spectral theory}, Cambridge Tracts in
  Mathematics, vol.~92, Cambridge University Press, 1990. 

\bibitem{MR1855839}
R.~deLaubenfels and H.~Emamirad, \emph{Chaos for functions of discrete and
  continuous weighted shift operators}, Ergodic Theory Dynam. Systems
  \textbf{21} (2001), no.~5, 1411--1427. 

\bibitem{MR2020387}
R.~deLaubenfels, H.~Emamirad, and K.-G. Grosse-Erdmann, \emph{Chaos for
  semigroups of unbounded operators}, Math. Nachr. \textbf{261/262} (2003),
  47--59. 

\bibitem{MR1468101}
Wolfgang Desch, Wilhelm Schappacher, and Glenn~F. Webb, \emph{Hypercyclic and
  chaotic semigroups of linear operators}, Ergodic Theory Dynam. Systems
  \textbf{17} (1997), no.~4, 793--819. 

\bibitem{MR1046376}
Robert~L. Devaney, \emph{An introduction to chaotic dynamical systems}, second
  ed., Addison-Wesley Studies in Nonlinearity, Addison-Wesley Publishing
  Company Advanced Book Program, Redwood City, CA, 1989. 

\bibitem{MR2152389}
Samir El~Mourchid, \emph{On a hypercyclicity criterion for strongly continuous
  semigroups}, Discrete Contin. Dyn. Syst. \textbf{13} (2005), no.~2, 271--275.
 

\bibitem{MR2280828}
\bysame, \emph{The imaginary point spectrum and hypercyclicity}, Semigroup
  Forum \textbf{73} (2006), no.~2, 313--316. 

\bibitem{MR1721989}
Klaus-Jochen Engel and Rainer Nagel, \emph{One-parameter semigroups for linear
  evolution equations}, Graduate Texts in Mathematics, vol. 194,
  Springer-Verlag, New York, 2000, With contributions by S. Brendle, M.
  Campiti, T. Hahn, G. Metafune, G. Nickel, D. Pallara, C. Perazzoli, A.
  Rhandi, S. Romanelli and R. Schnaubelt. 

\bibitem{MR954385}
Ramesh Gangolli and Veeravalli~S. Varadarajan, \emph{Harmonic analysis of
  spherical functions on real reductive groups}, Ergebnisse der Mathematik und
  ihrer Grenzgebiete [Results in Mathematics and Related Areas], vol. 101,
  Springer-Verlag, Berlin, 1988. 

\bibitem{MR884467}
Robert~M. Gethner and Joel~H. Shapiro, \emph{Universal vectors for operators on
  spaces of holomorphic functions}, Proc. Amer. Math. Soc. \textbf{100} (1987),
  no.~2, 281--288.

\bibitem{MR1111569}
Gilles Godefroy and Joel~H. Shapiro, \emph{Operators with dense, invariant,
  cyclic vector manifolds}, J. Funct. Anal. \textbf{98} (1991), no.~2,
  229--269. 

\bibitem{MR2041336}
Karl~G. Grosse-Erdmann, \emph{Recent developments in hypercyclicity}, Seminar
  of {M}athematical {A}nalysis ({M}alaga/{S}eville, 2002/2003), Colecc.
  Abierta, vol.~64, Univ. Sevilla Secr. Publ., Seville, 2003, pp.~157--175.
  

\bibitem{MR1685272}
Karl-Goswin Grosse-Erdmann, \emph{Universal families and hypercyclic
  operators}, Bull. Amer. Math. Soc. (N.S.) \textbf{36} (1999), no.~3,
  345--381. 

\bibitem{MR514561}
Sigurdur Helgason, \emph{Differential geometry, {L}ie groups, and symmetric
  spaces}, Pure and Applied Mathematics, vol.~80, Academic Press Inc., New
  York, 1978. 

\bibitem{MR754767}
\bysame, \emph{Groups and geometric analysis}, Pure and Applied Mathematics,
  vol. 113, Academic Press Inc., Orlando, FL, 1984. 

\bibitem{MR1484673}
Gerd Herzog, \emph{On a universality of the heat equation}, Math. Nachr.
  \textbf{188} (1997), 169--171. 

\bibitem{JiWeb:nr}
Lizhen Ji and Andreas Weber, \emph{{$L^p$} spectral theory and heat dynamics of
  locally symmetric spaces}, arXiv:0810.0209.

\bibitem{MR2231625}
T.~Kalmes, \emph{On chaotic {$C\sb 0$}-semigroups and infinitely regular
  hypercyclic vectors}, Proc. Amer. Math. Soc. \textbf{134} (2006), no.~10,
  2997--3002 (electronic). 

\bibitem{MR2358980}
\bysame, \emph{Hypercyclic, mixing, and chaotic {$C\sb 0$}-semigroups induced
  by semiflows}, Ergodic Theory Dynam. Systems \textbf{27} (2007), no.~5,
  1599--1631. 

\bibitem{MR1224619}
Vitali~A. Liskevich and M.~A. Perel'muter, \emph{Analyticity of sub-{M}arkovian
  semigroups}, Proc. Amer. Math. Soc. \textbf{123} (1995), no.~4, 1097--1104.
  

\bibitem{MR0348538}
Michael Reed and Barry Simon, \emph{Tensor products of closed operators on
  {B}anach spaces}, J. Funct. Anal. \textbf{13} (1973), 107--124. 

\bibitem{MR518528}
Robert~J. Stanton and Peter~A. Tomas, \emph{Pointwise inversion of the
  spherical transform on {$L\sp{p}(G/K)$}, {$1\leq p<2$}}, Proc. Amer. Math.
  Soc. \textbf{73} (1979), no.~3, 398--404. 

\bibitem{MR705991}
Robert~S. Strichartz, \emph{Analysis of the {L}aplacian on the complete
  {R}iemannian manifold}, J. Funct. Anal. \textbf{52} (1983), no.~1, 48--79.
  

\bibitem{MR1250269}
Karl-Theodor Sturm, \emph{On the {$L\sp p$}-spectrum of uniformly elliptic
  operators on {R}iemannian manifolds}, J. Funct. Anal. \textbf{118} (1993),
  no.~2, 442--453. 

\bibitem{MR1016445}
Michael~E. Taylor, \emph{{$L\sp p$}-estimates on functions of the {L}aplace
  operator}, Duke Math. J. \textbf{58} (1989), no.~3, 773--793. 

\bibitem{MR924464}
Nicholas~Th. Varopoulos, \emph{Analysis on {L}ie groups}, J. Funct. Anal.
  \textbf{76} (1988), no.~2, 346--410. 

\bibitem{Weber:2008ve}
Andreas Weber, \emph{The ${L \sp p}$ spectrum of {R}iemannian products}, Arch.
  Math. (Basel) \textbf{90} (2008), 279 -- 283.

\bibitem{Weber:yq}
\bysame, \emph{Tensor products of recurrent hypercyclic semigroups}, J. Math.
  Anal. Appl. \textbf{351} (2009), no.~2, 603 -- 606.

\end{thebibliography}

\end{document}